\documentclass[notitlepage,11pt,reqno]{amsart}

\usepackage{amsopn,esint,nicefrac}
\usepackage[final]{hyperref}
\usepackage[T1]{fontenc}
\usepackage{ esint }

\usepackage[utf8]{inputenc}
\usepackage{apptools}
\AtAppendix{\counterwithin{lemma}{section}} 
\usepackage[margin=1in]{geometry}
\allowdisplaybreaks
\newcommand{\stkout}[1]{\ifmmode\text{\sout{\ensuremath{#1}}}\else\sout{#1}\fi}

\newcommand{\Rn}{\mathbb{R}^N}
 \newcommand{\De} {\Delta}
  \newcommand{\eps} {\varepsilon}

\usepackage{graphicx,enumitem,dsfont,upgreek}
 \usepackage[dvips]{epsfig}
 \usepackage[mathscr]{eucal}
\usepackage{amscd}
\usepackage{amssymb}
\usepackage{amsthm}
\usepackage{amsmath}
\usepackage{latexsym}
\usepackage{dsfont}
\usepackage{upref}
\usepackage{hyperref}

\usepackage{color}
\theoremstyle{plain}

\newtheorem{thm}{Theorem}[section]
\theoremstyle{plain}
\newtheorem{lemma}[thm]{Lemma}

\newtheorem{cor}[thm]{Corollary}

\theoremstyle{definition}
\newtheorem{defi}{Definition}[section]
\newtheorem{rem}{Remark}[section]
\newtheorem*{maintheorem*}{Main Theorem}
\newtheorem*{maincorollary*}{Main Corollary}
{%
\setcounter{enumi}{0}

\begin{enumerate}}%
{\end{enumerate} }

{%
\setcounter{enumi}{0}

\begin{enumerate}}%
{\end{enumerate} }

\newcommand{\abs}[1]{\ensuremath{\left|#1\right|}}
\newcommand{\Om}{\ensuremath{\Omega}}

\definecolor{darkgreen}{rgb}{0.09, 0.65, 0.27}
\newcommand{\cred}{\color{red}}
\newcommand{\cgr}{\color{darkgreen}}
\newcommand{\cb}{\color{blue}}

\newcommand{\na}{\nabla}
\newcommand{\Depq}{-\Delta_p u-\Delta_q u}

\newcommand{\dx}{\ensuremath{\, {\rm d}x}}

\numberwithin{equation}{section} \allowdisplaybreaks

\usepackage{cancel,pdfsync}

\title[Liouville properties for inequalities with $(p,q)$ Laplacian]{Liouville properties for differential inequalities \\ with $(p,q)$ Laplacian operator}

\begin{document}

\author{Mousomi Bhakta, Anup Biswas}
\address{Department of mathematics, Indian Institute of Science Education and Research Pune, Dr.\
Homi Bhabha Road, Pune 411008, India,
email: mousomi@iiserpune.ac.in, \, anup@iiserpune.ac.in}

\author{Roberta Filippucci}
\address{Dipartimento di Matematica e Informatica, Universit\'a degli Studi di Perugia, Via Vanvitelli 1, 06123 Perugia, Italy, email: roberta.filippucci@unipg.it}

\begin{abstract}
In this paper, we establish several Liouville-type theorems for a class of nonhomogenenous quasilinear inequalities.
In the first part,  we prove various Liouville results associated with nonnegative solutions to 
\begin{equation*}\tag{$P_s$}
\Depq\geq u^{s-1} \, \text{ in }\, \Om,
\end{equation*} 
where $1<q<p$, $s>1$ and $\Om$ is any exterior domain of $\Rn$.  In particular, we prove that for $q<N$, inequality $(P_s)$ does not admit any positive solution when $s<q_*$ and  $(P_s)$ admits a positive solution if  $s>q_*$,  where $q_*=\frac{q(N-1)}{N-q}$ is the Serrin exponent for the $q$-Laplacian. Further, we show that when $s=q_*$ and $p<s$ the only nonnegative solution to $(P_s)$ is the trivial solution. On the other hand, for $q\geq N$  we prove that $u\equiv 0$ is the only nonnegative solution for $(P_s)$ for any  $s>1$.

In the second part, we consider the inequality 
\begin{equation*}\tag{$P_{sm}$}
\Depq \geq u^s |\nabla u|^m \quad \text{ in }\Rn,
\end{equation*}
where $1<q<p$,  $N>q$ and $s, \, m\geq 0$.
We prove that, for $\{0\leq m\leq q-1\}\cup\{m>p-1\}$,
the only positive solution to $(P_{sm})$ is constant, provided $s(N-q)+m(N-1)<N(q-1)$. This, in particular, proves that if $\Om=\Rn$ then any nonnegative solution to $(P_s)$ with $1<q<N$ and $1<s<q_*$ 
is the trivial solution. To prove Liouville in the range $0\leq m<q-1$, we first prove an almost optimal lower estimate of any nonnegative supersolution of $(P_{sm})$ and then leveraging this estimate we prove Liouville result. To the best of our knowledge, this technique is completely new and provides an
 alternative approach to the capacity method of Mitidieri-Pohozaev provided higher regularity is available.

\end{abstract}


\maketitle
\setcounter{tocdepth}{1}

\medskip

\noindent

\section{Introduction}

In recent years the study of nonexistence results for non-negative solutions, i.e. {\it Liouville property}, to quasilinear elliptic equations, which can be read as necessary conditions for the existence of solutions, has received considerable attention and significant progress has been achieved in this direction since the monumental result known as Liouville theorem for bounded analytic functions whose first statement was published by Cauchy in a short note in 1844, cfr. \cite{SZ} and \cite{M}. 

It is known that nonlinear Liouville theorems play also an important role in the study of existence results for Dirichlet boundary problems in bounded domains of $\mathbb R^N$ via the blow-up technique, developed by Gidas and Spruck in \cite{GS1}, based on the use a priori bounds for solutions to such equations derived from an application of a Liouville theorem.

Another powerful application of Liouville theorems is described in \cite{PQS}, where a general method for derivation of pointwise a priori estimates of  solutions, on an arbitrary domain and without any boundary conditions, is developed from Liouville type theorems. Additionally, \cite{PQS} explores a form of equivalence between the local properties of nonnegative solutions to superlinear elliptic problems and Liouville-type theorems.

In the prototype case of the $p$-Laplacian two seminal papers are given by
Mitidieri and Pohozaev \cite {MP} for supersolutions and by Serrin and Zou \cite{SZ} for solutions to 
\begin{equation}\label{pLapl}
    -\Delta_pu=u^{s-1}, \quad u\ge0,\quad\text{in } \mathbb R^N.
\end{equation}
Precisely, in \cite {MP} and \cite{SZ}, respectively for supersolutions and for solutions,  the following sharp thresholds  for Liouville property for $N>p$ are given 
$$1<s\le p_*,\quad   p_*=p(N-1)/(N-p) \quad\text{and}\quad 1<s< p^*,\quad   p^*=pN/(N-p).$$
For the special case of positive solutions when $p=2$, we refer to the celebrated work of Gidas and Spruck \cite{GS}.

The technique used in \cite {MP} is known as  the nonlinear capacity method and it is essentially based on two ingredients, the derivation of  a priori estimates based on the method of test functions and then an asymptotic for this estimates allow us to conclude the result.  In addition, it can be applied to functions  belonging solely to local Sobolev spaces, without reliance on a comparison/maximum principle or assumptions on the behavior at infinity. 

On the other hand, in \cite{SZ} the approach is completely different, precisely an extensive use of the comparison method together with an estimate from below for the solution far from zero yields the result.  We also point out the powerful method of Cutri and Leoni \cite{CL00}, which is based on the maximum principle and also applies to nonlinear operators of Pucci type. This approach was later extended by Armstrong and Sirakov \cite{AS11} and Alarc\'on et al.\ \cite{AGQ} to establish the Liouville property for nonnegative supersolutions in the presence of gradient nonlinearity. For a comprehensive overview of recent developments in the Liouville property, we refer the interested reader to the survey \cite{CG23}.

Moving to exterior domains, that is, domains $\Omega$ satisfying $\Omega \supset \{x : |x| > R\}$ for some $R > 0$,
 Bidaut-V$\acute{\rm e}$ron in \cite{BV} showed that \eqref{pLapl} admits only the trivial solution $u\equiv0$ whenever $N>p$ and $p<s\le p_*$, or $N=p$ and $p<s<\infty$, while
  Bidaut-V$\acute{\rm e}$ron and Pohozaev \cite{BP} showed that any nonnegative supersolution of
  \eqref{pLapl} is the trivial solution $u\equiv0$ whenever $N>p$ and $1<s\le p_*$, or $N=p$ and $1<s<\infty$, and this result is sharp. For the Liouville property to nonnegative supersolutions of
  \eqref{pLapl} when $N<p$ and $s>0$, we refer to \cite[Theorem I$'$]{SZ}.

The Liouville property for the inequality involving a gradient nonlinearity of the form
\begin{equation}\label{pLapl_ge_grad}
-\Delta_p u \geq u^s |\nabla u|^m \quad\text{in } \mathbb{R}^N,
\end{equation}
was first  proved by Mitidieri and Pohozaev in \cite[Theorem 15.1]{MP2001} when the exponents belong to the subcritical range given by 
\begin{equation}\label{critic}
s(N-p)+m(N-1)<N(p-1), \quad s+m>p-1.
 \end{equation}
Later, Filippucci in \cite[Corollary 1]{F} included both the critical case $s(N-q)+m(N-1)=N(q-1)$  and the case $0<s\le p-1-m$, cfr. \cite[Corollaries 1.5, 1.6]{F09}.  For the case $m=0$ in \eqref{pLapl_ge_grad}, we refer to  \cite[Theorem 2]{MP}. All these results are obtained employing the nonlinear capacity method,  while the Hamilton-Jacobi case $s=0$ and $m=1$ or $s=0$ and $m$ subcritical was handled  by Bidaut-V$\acute{\rm e}$ron et al. in \cite[Theorem 2.1]{BVHV} for $p=2$, and for $p>1$ by  Chang et al. in \cite[Theorem 1.2]{CHZ}  by using suitable change of variables  combined with integral a priori estimates.  For a pioneering paper dealing with \eqref{critic} in its radial form for $p = 2$ we refer to 
\cite{CM} (see also \cite{EM96}).  We also mention the work of Burgos-P\'erez et al.\ \cite{BGQ}, where the authors investigate the existence and non-existence of supersolution for $p=2$. A recent work of Goffi \cite{Gof} establishes the non-existence of supersolution at the endpoint case 
$s=0$ and $m=\frac{N(p-1)}{N-1}$, by employing divergence theorem (see also, Theorem~\ref{T1.6} below).

The aim of this paper is to obtain Liouville type theorems for quasilinear inequalities  driven by  the $(p, q)$-Laplace operator, an operator which falls in the class of differential operator with unbalanced growth.
The initial idea to treat such operators
comes from Zhikov \cite{Z}, see also \cite{Ma},  who introduced such classes to provide models of strongly anisotropic materials. Problems involving the  $(p,q)$-Laplacian operator arises from reaction-diffusion systems of the form
$
u_{t}= \mbox{div}[A(u)\nabla u]+c(x,u), $
where  the function $u$ describes the density or concentration of multicomponent substances, while the differential term corresponds to the diffusion with  coefficient $A(u)=|\nabla u|^{ p-2}+|\nabla u|^{q-2}$, whereas the term $c(x,u)$ is the reaction and relates to sources and loss processes. 
Finally, also Born-Infeld equation, that appears in electromagnetism, electrostatics and electrodynamics, by using the Taylor formula, is driven by the multi-phase differential operator.

 On the other hand, $(p,q)$-Laplacian operator is a special case of double phase operator used to model physical phenomena like non-Newtonian fluids and composites with materials having different hardening exponents.

 In this paper we  establish the Liouville property
 for $(p,q)$-Laplacian differential inequalities in two direction: first with a nonlinearity $f$  of the form $f(u)=u^{s-1}$  and in an exterior domain,  then with $f(u,\nabla u)=u^s|\nabla u|^m$ and in the entire $\mathbb R^N$.


 Our analysis encounters several substantial challenges. The foremost difficulty lies in the presence of a nonlinear, nonhomogeneous operator given by a combination of two operators, each either singular or degenerate,  with distinct growth behaviors. This structural complexity precludes the application of standard techniques that are typically effective in the homogeneous setting, even when dealing with purely power-type nonlinearities. Moreover, the presence of a gradient term combined with non-homogeneity introduces additional complications, particularly when building a  barrier-type function  playing the role of a lower bound for the solution, which demands the development of a sophisticated and nontrivial iteration scheme as well as refined estimates, not required in the $p$-Laplacian case, as it will be clear in the proof Lemma \ref{L5.1}.

Throughout the paper, we assume $1<q<p$.

The results obtained highlight the dominant influence of the lower-order term, that is 
$q$-Laplacian operator, which, in some sense, governs the structure of the analysis.

Before starting with the main theorems, we give the following definitions.

\begin{defi}
Let $s>1$. We say  that a non-negative function $u\in C(\Omega)\cap  W^{1,p}_{loc}(\Om)$ is a solution of the differential inequality $-\Delta_p u-\Delta_q u\geq u^{s-1}, \quad u\geq 0\, \text{ in  }\, \Omega$ if
$$\int_{\Om}|\nabla u|^{p-2}\nabla u\cdot\na \varphi \dx+\int_{\Om}|\nabla u|^{q-2}\nabla u\cdot\na \varphi \dx\geq \int_{\Om}u^{s-1}\varphi \dx\quad\forall\,\, \varphi\in C^1_c(\Om), \quad\varphi\geq 0.$$
\end{defi}

\begin{defi}
We say that a non-negative function $u\in C^1(\Om)$ is a solution  of the differential inequality $-\Delta_p u-\Delta_q u\geq f(x,u,\na u), \quad u\geq 0\, \text{ in  }\, \Omega$ if 
$$\int_{\Om}|\nabla u|^{p-2}\nabla u\cdot\na \varphi \dx+\int_{\Om}|\nabla u|^{q-2}\nabla u\cdot\na \varphi \dx\geq \int_{\Om}f(x,u,\na u)\varphi \dx\quad\forall\,\, \varphi\in C^1_c(\Om),\quad\varphi\geq 0,$$
\end{defi}
with $f$ given by above. Now we state the main theorems of this paper. 
\begin{thm}\label{t:qbig}
Let $p>q\geq N$ and $u$ be any nontrivial non-negative solution   to  $$-\Delta_p u-\Delta_q u\geq 0\quad\text{ in } \Rn.$$
Then $u$ is constant. 
\end{thm}
The above theorem generalizes \cite[Theorem {16.1}]{MP2001},  see also  \cite[Theorem 2]{MP_Bern}, \cite[Theorem 2.6]{BD02} and \cite[Theorem II-(a)]{SZ}, 
 from 
$p$-superharmonic functions to $(p,q)$-superharmonic functions  in the class $u\in C(\Omega)\cap  W^{1,p}_{loc}(\Om)$.  Liouville-type theorems of the above form have been studied in a much broader context and established using various methods, namely: via Hadamard three circle
theorem in \cite{PW-book},
Caccioppoli estimates in \cite{Moon}, stochastically incompleteness property in \cite{Gri}, and the divergence theorem in \cite{RS01} (see also \cite{PRS08} for several associated equivalence relations).

\begin{rem}
When $N\leq q<p$, the inequality $-\Delta_p u-\Delta_q u\geq 0$ has the (bounded) positive non-constant
solution $u=1-\frac{1}{|x|}$ in the exterior domain $\{|x|>1\}$, which indicates the necessity of considering Theorem~\ref{t:qbig} on the entire space $\Rn$.
\end{rem}

Our next result deals with the general exterior domains and gives a lower estimate from below for  $(p,q)$-superharmonic functions.  For homogeneous elliptic operators similar lower bound estimate can also be found in \cite{AS11a,AS11}.

\begin{lemma}\label{L1.5}
Let $\Omega$ be an exterior domain  and $u$ be a positive solution to the differential inequality $$\Depq\geq 0 \quad\text{ in }\, \Omega.$$ Then the following holds:

(i) If $1<q<N$ then for any positive $\theta$ satisfying
$\theta\geq \frac{N-q}{q-1}$,  there exists a constant $\kappa>0$ such that $$u(x)\geq \kappa|x|^{-\theta}
\quad\text{ in }\, |x|>R,$$
where $\kappa$ depends on $\theta, \Om$ and $u$.

(ii) If $q\geq N$ then $$\lim\inf_{|x|\to\infty}u(x)>0.$$
\end{lemma}

 In the next result we prove that if $u$ is a nonnegative supersolution to 
$\Depq = u^s |\na u|^m$ in an exterior domain $\Om$ then we can further improve the lower estimate of $u$. 

\begin{lemma}\label{L5.1}
Let  $1<q<N$, $s\geq 0$, $0\leq m\leq  q-1$ and 
\begin{equation}\label{L5.1A}
s(N-q)+m(N-1)<N(q-1).
 \end{equation}
Let $\Omega$ be an exterior domain and let   $u\in C(\bar\Om)\cap C^1(\Om)$ be a positive solution  to
\begin{equation}\label{grad-non-Om}
\Depq \geq u^s |\na u|^m \quad \mbox{in }\; \Om.
\end{equation}
Then, for any $\theta>0$,  there exists a constant $C>0$ such that
\begin{equation}\label{L5.1B}
u(x)\geq C |x|^{-\theta}\quad \text{in}\;\; \bar\Om\cap\{|x|\geq 1\},
\end{equation}
where $C$ depends on $\theta, \Om$ and $u$.
\end{lemma}

The above lemma is one of the key ingredients in
proving Theorems~\ref{t:s2} and ~\ref{T1.7}. Proof of Lemma~\ref{L5.1} follows an iteration method where we improve the lower bound  for $u$ in every iteration. Interestingly,  \eqref{L5.1A} becomes crucial for us to initiate the iteration. 
Furthermore, we point out that \eqref{L5.1A} represents the subcritical range (cfr. \eqref{critic} ) respect to $q$.

\begin{thm}\label{t:s1}
Let $\Omega$ be an exterior domain in $\Rn$ with 
$$q<N \quad\text{and}\quad 1<q<p<s\le  q_*=\frac{q(N-1)}{N-q}.$$ Then, any nonnegative solution $u$ of the differential inequality  
\begin{equation}\label{main1}
\Depq \geq u^{s-1}
\quad\text{ in }\, \Omega
\end{equation}
is  the trivial solution $u\equiv 0$.
\end{thm}

\begin{rem}\label{r:1'}
We first point out that  conditions on the exponents assumed in  Theorem~\ref{t:s1}(i) produce an upper bound for $N$ given by $N<\frac{q(p-1)}{p-q}$.  However, thanks to Theorem~\ref{t:s2} we see that the upper bound on $N$ is actually not applicable when $s<q_*$

Furthermore, Theorem~\ref{t:s1} generalizes  the result of Bidaut-V\'eron and Pohozaev \cite[Theorems 3.3 (iii) and 3.4 (ii)]{BVP} (also see \cite[Theorem I]{SZ}) from $p$-Laplacian operator to $(p,q)$-Laplacian operator. Note that, in the case of $(p,q)$-Laplacian operator, the critical exponent $q_*$ remains the same as it was for $q$-Laplace operator. In the next remark we show that $q_*$ is indeed the critical exponent for \eqref{main1}.

\end{rem}
\begin{rem}\label{r:1}
In Theorem~\ref{t:s1}, $q_*$ is the critical exponent for \eqref{main1} in the sense that for  $t>q_*$, there exists $R_0>0$ and a positive function $\varphi$ satisfying \eqref{main1} in $|x|>R_0$.

To see this,  consider $N>q$, $t>q_*$ and $\theta_*=\frac{N-q}{q-1}$. Thus, $\theta_*>\frac{N-p}{p-1}$. Further, since $$-\theta_*(q-1)-q=-\theta_*(q_*-1)>-\theta_*(t-1),$$ we can choose
$\theta\in(0, \theta_*)$ such that $-\theta(q-1)-q>-\theta(t-1)$ and $\theta>\frac{N-p}{p-1}$. Now, letting $$\varphi(x)=k|x|^{-\theta}$$ it is easy to check that
(see the proof of Lemma~\ref{L1.5})
$$-\Delta_p \varphi(x)-\Delta_q \varphi(x)=  k^{p-1}c_{n, p, \theta} |x|^{-\theta(p-1)-p}+k^{q-1} c_{n, q,\theta} |x|^{-\theta(q-1)-q} .$$
for some constants $c_{n, p,\theta}$ and $c_{n, q,\theta}$. Moreover, $c_{n, q,\theta}>0>c_{n, p,\theta}$. We choose $k>0$ such that $k^{q-1} c_{n, q,\theta}\geq 2$. 
As $\theta(p-1)+p>\theta(q-1)+q$, we can find $R_0\geq 1$ large enough
so that 
$$-\Delta_p \varphi(x)-\Delta_q \varphi(x)\geq k^{q-1}\frac{c_{n,q,\theta}}{2} |x|^{-\theta(q-1)-q}\quad \text{for}\; |x|\geq R_0.$$
Also, from our choice of $\theta$, we have $|x|^{-\theta(q-1)-q}\geq |x|^{-\theta(t-1)}$ for $|x|\geq 1$. Hence
$$-\Delta_p \varphi(x)-\Delta_q \varphi(x)\geq k^{q-1}\frac{c_{n,q,\theta}}{2} \varphi^{t-1}\geq \varphi^{t-1}\quad \text{for}\; |x|\geq R_0.$$

\end{rem}

Combining Theorem~\ref{t:s1} and Remark~\ref{r:1},  the following corollary holds.

\begin{cor} If 
$\Om$ is an exterior domain, $1<q<N$ and $q<p<s$, then the differential inequality $\Depq \geq u^{s-1}$ has a positive solution if and only if  $s>q_*$.
\end{cor}

\begin{thm}\label{t:Nq}
Let $\Om$ be an exterior domain and let $N\leq q$. Then the differential
inequality \eqref{main1} has only the trivial solution $u=0$ provided $s\in(1,\infty)$. 
\end{thm}
For $p=q$, this result is can be found in Serrin and Zou 
\cite[Theorem~I$'$]{SZ}. 

Our next result deals with differential inequality \eqref{main1} in the range of $s$ which are not covered in Theorem~\ref{t:s1}.  In particular, we present a new and alternative proof based on Lemma~\ref{L5.1}, which eliminates the lower bound on $s$
 required by Theorem~\ref{t:s1}. However, the critical case $s=q^*$
 remains out of reach.

\begin{thm}\label{t:s2}  Let $q<N$, 
$\Omega$ be an exterior domain  and $1<s<q_*$. Then the only nonnegative solution to \eqref{main1} is $u\equiv 0$.
\end{thm}

It is worth noting that Theorem~\ref{t:s2} improves Theorem~\ref{t:s1} by proving the nonexistence of nonnegative solutions to \eqref{main1} for the range $\{1 < s < q_* \leq p\}\cup \{1<q<s<p<q_*\}\cup\{1<s<q<p\leq q_*\}$ and provides an alternative proof to Theorem~\ref{t:s1} when $1<q<p<s<q_*$. As a matter of fact the proof of Theorem~\ref{t:s2} offers an alternative to the classical nonlinear capacity method by Mitidieri and Pohozaev \cite{MP}, provided that higher regularity for solutions is available.
 

In particular, for the case $p=q$ and $s>p$, \cite{BP,SZ} employs the capacity method of Mitidieri and Pohozaev,  similar to the proof of Theorem~\ref{t:s1}. For $1<s\leq p$, however, the argument of \cite{SZ} crucially depends on the existence of positive eigenfunctions and invariance of the equation under scaling, a consequence of the operator's homogeneity (see Lemma 2.8 therein).  We also mention the works \cite{AS11a, AS11}, where similar problems are studied for a family of operators satisfying a general set of hypotheses. The authors employ maximum principle and weak Harnack estimates to establish their results. However, the arguments in these articles require either homogeneity of the operator or the validity of a weak Harnack inequality at  {\it small scale}. Given the structure of our operator, the applicability of their methods in the present context remains unclear.

In our present case, which is non-homogeneous, the nonlinear capacity method remains applicable for $s>p$ but breaks down for $s\leq p$. Furthermore, the lack of homogeneity prevents us from using a positive eigenfunction (see the work of \cite{BF23} for eigenfunctions of this operator). To circumvent this difficulty, we instead rely on Lemma~\ref{L5.1}, thereby providing a new and alternative method in Theorem~\ref{t:s2} for establishing the Liouville property for \eqref{main1}.

In the next theorem, we study differential inequality with $(p,q)$-Laplace operator in presence of nonlinear gradient terms.
\begin{thm}\label{T1.6}
Suppose $N\geq 2$,  $1<q<N$, $s\geq 0$, $m> p-1$ and 
\begin{equation}\label{T1.8AA}
 s(N-q)+m(N-1)\leq N(q-1)
\end{equation}
 holds true. Let $u$ be any positive solution of the inequality 
\begin{equation}\label{grad non}
 \Depq \geq u^s |\nabla u|^m \quad \text{ in }\Rn.
 \end{equation}
Then $u$ is constant.
\end{thm}

We note that our proof of Theorem~\ref{T1.6} employs a different and considerably simpler approach. It directly applies the nonlinear capacity method, coupled with a lower-bound estimate for the solution. However, this direct method necessitates the condition $m > p - 1$.

We now elaborate on the difficulties encountered when attempting to adapt the method from \cite[Theorem 1.2]{CHZ}. The well-known technique used there relies on the substitution $u = v^b$ for a chosen constant $b$. This transformation converts the inequality $-\Delta_p u \geq u^s |\nabla u|^m$ into a form amenable to the methods of \cite{MP}, specifically $-\Delta_p v \geq |\nabla v|^\alpha$ for a suitable exponent $\alpha$.

If we apply the same substitution $u = v^b$ to our inequality \eqref{grad non}, the result is an inequality with purely gradient terms on the right-hand side, namely
\begin{align*}
&-b\big(\Delta_q v+|b|^{p-q}v^{(b-1)(p-q)}\Delta_p v\big)\geq C|\nabla v|^\alpha, \quad\text{if }\, s+m-q+1\neq 0,\\
&-\Delta_q v-e^{(p-q)v}\Delta_p v\geq C|\nabla v|^{\alpha}, \quad\text{if }\, s+m-q+1= 0,
\end{align*}
where $q-1<\alpha<\frac{N(q-1)}{N-1}$, $b(b-1)>0$. Crucially, due to the presence of the $(p, q)$-Laplacian, the resulting expression is not simpler than the original. The coupled operators $\Delta_q v$ and $\Delta_p v$ complicate the analysis.
 Another possibility could be to apply divergence theorem, as in \cite{Gof,RS01}, but due to the non-homogeneous nature of our operator, the end first order differential inequality becomes untractable due to inhomogeneous nature of the operator. For a similar reason, we do not follow the approach of \cite{BGQ} which relies on the construction of suitable radial solutions, by solving appropriate first order equations.

Therefore, to circumvent this difficulty, the proof of Theorem~\ref{T1.6} instead adapts the foundational capacity approach from \cite{MP} and integrates it with the lower-bound estimate provided in Lemma~\ref{L1.5}.

 In particular, Theorem~\ref{T1.6} extends previous Liouville results in \cite{MP2001}, \cite{BVHV} and \cite{CHZ}, as described in details below  inequality \eqref{pLapl_ge_grad} Though our Theorem~\ref{T1.6} gives Liouville results only for $m>p-1$, we can still recover the Liouville results for $0\leq m\leq q-1$, as mentioned below.

\begin{thm}\label{T1.7}
Suppose that $s\geq 0$,   $0\leq m\leq q-1$, $q<N$ and \eqref{L5.1A} is satisfied. 
Then any positive solution to
\begin{equation}\label{T1.7eq}
\Depq \geq u^{s}|\na u|^m\quad \text{in}\quad \Rn,
\end{equation}
is a constant.
\end{thm}
The proof of Theorem~\ref{T1.6} relies critically on the assumption $m > p-1$, which is essential for bounding certain integrals. This reliance prevents the same technique from being applied to the case $m \leq q-1$. Furthermore, the alternative strategy of employing an auxiliary function, discussed in the previous page turns out to be unhelpful for this problem.
To address this issue, we introduce a new technique. Its first step is to substantially improve the known lower bound for solutions $u$ to \eqref{T1.7eq}, a result we prove in Lemma~\ref{L5.1}. This improved lower bound is pivotal, as it allows us to successfully implement a maximum principle argument to establish Theorem~\ref{T1.7}. We are confident that this core idea will be valuable for deriving Liouville theorems in other contexts.
 In particular, Theorem~\ref{T1.7} extends previous Liouville results in \cite{F09}, as illustrated below  \eqref{pLapl_ge_grad}.
 However, the range $q-1<m\leq p-1$ remains open to explore for \eqref{T1.7eq}.

\smallskip

While in this paper we are dealing only with differential inequalities, in our fourth-coming paper \cite{BBF}, we have proved the Liouville results associated with quasilinear elliptic equations with $(p,q)$-Laplace operators in presence of various type of nonlinear gradient terms. 

\smallskip

The paper is organized as follows. In Section 2, we prove some comparison principles. Section 3 deals with $(p,q)$-super harmonic functions, namely the proof of Theorem~\ref{t:qbig} and Lemma~\ref{L1.5}. In Section 4, we give the proof of Theorem~\ref{t:s1} and Theorem~\ref{t:Nq}. Finally, Section 5 deals with the proof of Theorem~\ref{T1.6}, Theorem~\ref{T1.7} and Theorem\ref{t:s2}.

\section{Preliminary}

\begin{lemma}\label{l:comp1}
Let $\Omega$ be a  domain of $\Rn$ and $u,\, v$ be two continuous functions in the Sobolev space $W^{1,p}_{loc}(\Om)$ satisfy the inequality in the distributional sense
$$\Depq \geq f, \quad -\De_p v-\De_q v\leq f \quad\text{ in }\, \Omega,$$
{where $f\in L^{\frac{pN}{pN+p-N}}_{loc}(\Om)$}.
Suppose that, $u\geq v$ in $\partial\Om$ in the sense that the set $\{u-v+\eps\leq 0\}$
has compact support in $\Om$ for every $\eps>0$. Then $u\geq v$ in $\Om$.
\end{lemma}
\begin{proof}
 Let $\varepsilon\in (0, 1)$. Define $w=(v-\eps-u)_+$. From our assertion above we have $w\in W^{1, p}_0(\Om_1)$ where
$\Om_1\Subset \Om$. Using $w$ as a test function in the given inequations and denoting $v_\eps=v-\eps$, we obtain
\begin{align*}
0 &\leq \int_{\Om_1} (|\na u|^{p-2}\na u-|\na v_\eps|^{p-2}\na v_\eps)\cdot \nabla w \dx +
\int_{\Om_1} (|\na u|^{q-2}\na u-|\na v_\eps|^{q-2}\na v_\eps)\cdot \nabla w \dx
\\
&=  \int_{\Om_1\cap\{w>0\}} (|\na u|^{p-2}\na u-|\na v_\eps|^{p-2}\na v_\eps)\cdot (\nabla v_\eps-\nabla u) \dx
\\
&\qquad+\int_{\Om_1\cap\{w>0\}} (|\na u|^{q-2}\na u-|\na v_\eps|^{q-2}\na v_\eps)\cdot (\nabla v_\eps-\nabla u) \dx.
\end{align*}
Therefore, 
\begin{align*}
& \int_{\Om_1\cap\{w>0\}} (|\na u|^{p-2}\na u-|\na v_\eps|^{p-2}\na v_\eps)\cdot (\nabla u-\nabla v_\eps) \dx
\\
&\qquad+\int_{\Om_1\cap\{w>0\}} (|\na u|^{q-2}\na u-|\na v_\eps|^{q-2}\na v_\eps)\cdot (\nabla u-\nabla v_\eps) \dx\leq 0.
\end{align*}

Case I:  $q\geq 2$.

Then $p\geq 2$ and using the inequality (see \cite[(I), pg.73]{L})
$$\langle|b|^{p-2}b-|a|^{p-2}a, b-a\rangle\geq 2^{2-p}|b-a|^{p},$$
we obtain
$$2^{2-p}\int_{\Om_1\cap\{w>0\}} |\na u-\na v_\eps|^p \dx +2^{2-q} \int_{\Om_1\cap\{w>0\}}  |\na u-\na v_\eps|^q \dx\leq 0.$$

Case II: $q<p\leq 2$.

In this case using the inequality (see \cite[(VII), pg.76]{L})
$$\langle|b|^{p-2}b-|a|^{p-2}a, b-a\rangle\geq (p-1)|b-a|^{2}(1+|b|^2+|a|^2)^\frac{p-2}{2},$$
we obtain
\begin{align*}
&(p-1)\int_{\Om_1\cap\{w>0\}} |\na u-\na v_\eps|^2(1+|\na u|^2+|\na v_\eps|^2)^\frac{p-2}{2} \dx \\
&\qquad+(q-1)\int_{\Om_1\cap\{w>0\}} |\na u-\na v_\eps|^2(1+|\na u|^2+|\na v_\eps|^2)^\frac{q-2}{2} \dx\leq 0.
\end{align*}

Case III: $q<2\leq p$
\begin{align*}
&2^{2-p}\int_{\Om_1\cap\{w>0\}} |\na u-\na v_\eps|^p \dx \\
&\qquad+(q-1)\int_{\Om_1\cap\{w>0\}} |\na u-\na v_\eps|^2(1+|\na u|^2+|\na v_\eps|^2)^\frac{q-2}{2} \dx\leq 0.
\end{align*}

Thus from  all the above three cases, it follows that $\na w=0$ on $\{w>0\}$, implying that $w=0$ in $\Om_1$. Hence $v-\eps\leq u$ in $\Om$. Since $\eps$ is arbitrary, the conclusion follows.

%

\end{proof}

We also need the following strong maximum principle.
\begin{lemma}\label{R1}
If $u$ is a nontrivial nonnegative solution of the differential inequality 
\begin{equation}\label{-pq_ge0_Omega}-\Delta_p u-\Delta_q u\geq 0\quad \text{in}\;\; \Omega,\end{equation} 
then $u>0$ in $\Om$.
\end{lemma}

\begin{proof}
This can be done following an argument similar to $\Delta_p$, see \cite{Vaz}. We split the proof in two parts.

First, assume that $u\in C^1(\Omega)$. Now for any $r\in (0, 1)$ we consider the function 
$\zeta(x)= e^{-\alpha |x|}-e^{-\alpha r}$ for $\frac{r}{2}\leq |x|\leq r$, and $\alpha>0$ will be chosen later. A straightforward calculation reveals
\begin{align*}
\Delta_p\zeta(x)&= \alpha^{p-1} e^{-\alpha(p-1)|x|}\left[\alpha(p-1)-\frac{n-1}{|x|}\right],
\\
\Delta_q\zeta(x)&= \alpha^{q-1} e^{-\alpha(q-1)|x|}\left[\alpha(q-1)-\frac{n-1}{|x|}\right].
\end{align*}
We can choose $\alpha$ large enough, depending on $r$, such that $-\Delta_p\zeta-\Delta_ q\zeta<0$ for $\frac{r}{2}\leq |x|\leq r$.

Now suppose, $\mathcal{Z}:=\{x\in\Omega\;:\; u(x)=0\}$ in nonempty.  In particular, $\mathcal{Z}$ is relatively closed in $\Omega$. Consider a ball $B_r(x_0)\Subset\Om$ that touches
$\mathcal{Z}$ from outside, this condition is known as "exterior sphere condition". Choose $\kappa>0$ small enough so that $\kappa\zeta_{x_0}(x):=\kappa\zeta(x-x_0)\leq u$ on $|x|=\frac{r}{2}$. Also, it is easily seen that
$-\Delta_p(\kappa\zeta_{x_0})-\Delta_p(\kappa\zeta_{x_0})<0$ for $\frac{r}{2}\leq |x-x_0|\leq r$. Applying the comparison principle
Lemma~\ref{l:comp1}, it readily
gives us, for $\kappa>0$ small enough,
$$u(x)\geq \kappa\zeta_{x_0}(x)\quad\text{for}\quad  \frac{r}{2}\leq |x-x_0|\leq r.$$ Let 
$z\in \mathcal{Z}\cap  \partial{B}_r(x_0)$. Then, we have
$$u(x)-u(z)\geq \kappa \left(e^{-\alpha|x-x_0|}-e^{-\alpha r}\right)\geq \kappa \alpha e^{-\alpha r} (r-|x-x_0|).$$
This is a standard quantitative Hopf's inequality  implying that $\nabla u(z)\neq 0$. Indeed, for $\nu=\frac{x_0-z}{|x_0-z|}$ and noting that $|z-x_0|=r$,
$$\frac{\partial u}{\partial \nu}(z)=\lim_{t\to0^+}\frac{u(z+t\nu)-u(z)}t\ge \kappa  \alpha e^{-\alpha r}\lim_{t\to0^+}  \frac{r-|z+t\nu-x_0|}t =\kappa  \alpha e^{-\alpha r}>0.$$ But, since $u\in C^1(\Om)$ and $u(z)=0$,  that is $z$ is a minimum point in $\Omega$ for $u$,  we must have
$\nabla u(z)=0$. This is a contradiction, proving that $\mathcal{Z}$ is empty.

Now we suppose that $u\in C(\Om)\cap W^{1, p}_{loc}(\Om)$. Since $u\not\equiv 0$, there exists a ball $B_\delta(\bar{x})\Subset\Om$ such that
$u\neq 0$ on $\partial B_\delta(\bar{x})$. Let $v\in u + W^{1, p}_0(B_\delta(\bar{x}))$ be a minimizer of
$$v\mapsto \frac{1}{p}\int_{B_\delta(\bar{x})} |\nabla u|^p \dx + \frac{1}{q}\int_{B_\delta(\bar{x})} |\nabla u|^q \dx,$$
over the set $u+ W^{1, p}_0(B_\delta(\bar{x}))$. In particular, $v$ is a weak solution to
\begin{equation}\label{comp_B_delta}-\Delta_p v-\Delta_p v =0 \quad \text{in} \; B_\delta(\bar{x}), \quad v=u\quad \text{on}\; \partial B_\delta(\bar{x}).\end{equation}
From \cite[Theorem~1.7]{Lib91} we know that $v\in C^1(B_\delta(\bar{x}))$, then by \eqref{-pq_ge0_Omega} combined with  \eqref{comp_B_delta} and by comparison principle we have $0\leq v\leq u$ in $B_\delta(\bar{x})$.
Since $u\not\equiv 0$ on $\partial B_\delta(\bar{x})$, we also have $v\not\equiv 0$. Thus, from the first part we obtain that $v>0$ in 
$B_\delta(\bar{x})$, implying $u>0$ in $B_\delta(\bar{x})$. Now the same argument applies to every ball that interests $B_\delta(\bar{x})$, and
then to every ball that intersects one of those balls and so on. This completes the proof.
\end{proof}

We also recall the following result from \cite[Lemma 3]{BBF} which will be used in proving Theorem~\ref{T1.7}.
\begin{lemma}\cite[Lemma 3]{BBF}\label{L-vis}
Suppose that $u\in C^1(\Om)$ is local weak supersolution to $\Depq \geq g$ in $\Om$, where $g$ is a continuous functions. 
If for some point $x\in\Omega$, 
there exists a function $\varphi\in C^2(B_r(x))$, $B_r(x)\Subset \Omega$, such that $\nabla\varphi(x)\neq 0$ and $u(y)-\varphi(y)\geq u(x)-\varphi(x)$ for $y\in B_r(x)$, then we have $-\De_p \varphi(x) -\De_q\varphi(x) \geq g(x)$.
\end{lemma}

\section{$(p,q)$ superharmonic functions}

First, we prove non-existence of nontrivial super-solution for $N\leq q$, namely Theorem~\ref{t:qbig}.

\medskip

{\bf Proof of Theorem~\ref{t:qbig}}.
Let $u$ be a nonnegative nontrivial solution to the above inequation. Therefore, by Remark~\ref{R1}, $u>0$ in $\Rn$. Now, we define $\psi(x)=\psi(|x|):=-\log |x|$ for $x\neq 0$. It is then easy to see that
$$ \De_p \psi(x)=-\frac{N-p}{|x|^p} \quad\text{and}\quad \De_q\psi(x)=-\frac{N-q}{|x|^q}\quad \text{for}\; x\neq 0.$$
Now fix a point $x_0\in \Rn$ and let $\kappa_m=\min_{|x-x_0|\leq \frac{1}{m}} u$. For $\varepsilon\in (0, 1)$, we define
$$v=\kappa_m + \varepsilon\bigg(\psi(x-x_0)- \psi\big(\frac{1}{m}\big)\bigg).$$ It is evident that $-\De_p v-\De_q v\leq 0$ in $\big(B(x_0, \frac{1}{m})\big)^c$ and $u\geq v$ on $|x-x_0|=\frac{1}{m}$. Since $v(x)\to -\infty$ as
$|x|\to\infty$, clearly $\exists R>>1$ such that $u\geq v$ in $B(x_0, R)^c$. Now applying 
 the comparison principle in $\big(B(x_0, \frac{1}{m})\big)^c\cap B(x_0, R)$ we obtain 
 $u\geq v$ in $\big(B(x_0, \frac{1}{m})\big)^c$, for all $\varepsilon\in (0, 1)$ and $\forall\,m>0$. Now, we first let $\varepsilon\to 0$ and then $m\to\infty$,
to get $u(x)\geq u(x_0)$ for all $x\in\Rn$. Since $x_0$ is an arbitrary point, $u$ must be a constant. \hfill $\Box$

\begin{rem}
From the proof of Theorem~\ref{t:qbig} we can relax the non-negativity assumption on $u$ by a logarithmic decay to $-\infty$. More precisely, the conclusion of Theorem~\ref{t:qbig} remains valid for supersolution $u$ satisfying
$$\liminf_{|x|\to\infty} \frac{u(x)}{\log|x|}\geq -\varepsilon,$$
for any $\varepsilon>0$
Compare this with \cite[Remark~(i), p. 130]{PW-book}
\end{rem}

\medskip

Next, we move to the exterior domain and prove Lemma~\ref{L1.5}.

{\bf Proof of Lemma~\ref{L1.5}}:

{\bf Case 1}: $1<q<N$.

Define 
$$\varphi_\theta(x)=|x|^{-\theta}\quad \text{for}\; x\neq 0.$$
Then, a straightforward computation yields that
\begin{align*}
-\De_p\varphi_\theta(x) & = \theta^{p-1} |x|^{-(\theta+1)(p-1)-1}\big(N-(\theta+1)(p-1)-1\big),
\\
-\De_q\varphi_\theta(x) & = \theta^{q-1} |x|^{-(\theta+1)(q-1)-1}\big(N-(\theta+1)(q-1)-1\big),
\end{align*}
for $x\neq 0$. 

Since $\theta\geq \frac{n-q}{q-1}>\frac{n-p}{p-1}$, 
it follows that 
$$ -\De_p\varphi_\theta\leq 0\quad \text{and}\quad -\De_q\varphi_\theta\leq 0.$$
Define, $$k=R^{\theta}\min_{|x|=R}u>0$$
and $$v=\frac{k}{|x|^\theta}.$$
Thus, $-\De_p v-\De_q v\leq 0$ in $\Om$. Now let $v_\varepsilon=v-\varepsilon$. Since $\theta>0$, there exists
$r_\varepsilon$ large enough so that $v_\varepsilon(x)\leq 0$ for $|x|\geq r_\varepsilon$. From our choice above, we have
$u\geq v_\varepsilon$ in $\partial (B_{r_\varepsilon}\setminus B_{R})$. Therefore, from the comparison principle, we obtain
$u\geq v_\varepsilon$ in $B_{r_\varepsilon}\setminus B_{R}$. Since $v_\varepsilon\leq 0$ in $B^c_{r_\varepsilon}$, it follows that
$u\geq v_\varepsilon$ in $B^c_{R}$. Now letting $\varepsilon\to 0$, it follows $u\geq \frac{\kappa}{|x|^\theta}$ in $|x|>R$.

{\bf Case 2:} $q\geq N$.

In this case, for any $\theta>0$, we have $\theta>0\geq \frac{N-q}{q-1}>\frac{N-p}{p-1}$. Hence, 
$$-\Delta_p\varphi_\theta-\Delta_q\varphi_\theta\leq 0.$$
As in the case 1, we define $v_\eps$ and using the same argument we derive $u\geq v_\varepsilon=\frac{k}{|x|^\theta}-\eps$ in $B^c_{R}$. Since, this is true for any $\theta>0,\, \eps>0$, we first let $\eps\to 0$ and then let $\theta\to 0$ yields us
$$u\geq k \quad\text{in }\, |x|>R.$$
Hence, $\lim\inf_{|x|\to\infty}u(x)>0.$
\hfill $\Box$


\section{Power nonlinearity in exterior domains}

In this section, we prove Theorem~\ref{t:s1}. For this, we first prove the following key lemma,   taking inspiration from the $p$-Laplacian case by Serrin-Zou \cite{SZ}, see also \cite{ruiz}.

\begin{lemma}\label{L1.8}
Let $\Omega$ be an exterior domain, $1<q<p<s$ and $q < N$. Assume $u$ to be a positive  solution to \eqref{main1} and $R>0$. 

Then, for all  $\gamma\in (0, s-1)$, there exists a constant $C=C(N,p,q,s,\gamma)>0$ such that
\begin{equation}\label{EL1.8A}
\int_{B(x_0, R)} u^{\gamma} \dx\leq C R^{N-\frac{q\gamma}{s-q}} \end{equation}
for all  ${B(x_0, R)}$ such that it holds $B(x_0, 2R)\subset\Omega$. 
\end{lemma}

\begin{proof}
Let $\zeta\in C^2_c(B(0,2))$ be a radial cut-off function such that $\zeta\equiv 1$ in $B(0,1)$, $|\nabla \zeta|\leq 2$, $|D^2\zeta|\leq C$.  For $k>p$ to be determined later and 
$d=s-1-\gamma$, we define
$$\varphi:=\frac{\zeta^k\big(\frac{|x-x_0|}{R}\big)}{u^d} >0$$ as a test function for \eqref{main1}. This yields us

%
\begin{align*}
\int_{B_{2R}(x_0)}u^{s-1-d}  \zeta^k \dx \leq\int_{B_{2R}(x_0)} |\nabla u|^{p-2}\nabla u\cdot \nabla (u^{-d}\zeta^k) \dx + \int_{B_{2R}(x_0)} |\nabla u|^{q-2}\nabla u\cdot \nabla (u^{-d}\zeta^k)\dx,
\end{align*}
which can also be written as
\begin{align}\label{ET1.8A}
&\int_{B_{2R}(x_0)}u^{\gamma}  \zeta^k \dx + d \int_{B_{2R}(x_0)}  u^{\gamma-s}|\nabla u|^{p}\zeta^k \dx + d \int_{B_{2R}(x_0)}  u^{\gamma-s}|\nabla u|^{q}\zeta^k \dx\nonumber
\\
&\qquad \leq\int_{B_{2R}(x_0)} u^{-d}|\nabla u|^{p-2}\nabla u\cdot \nabla \zeta^k \dx + \int_{B_{2R}(x_0)} u^{-d}|\nabla u|^{q-2}\nabla u\cdot \nabla \zeta^k\dx.
\end{align}
As $|\na\zeta^k|=k\zeta^{k-1}|\na\zeta|\leq 2k\zeta^k \frac{1}{R\zeta}$, using 
 Young's inequality with the pair $(\frac{p}{p-1}, p)$
and writing $u^{-d}=u^{(\gamma-s)\frac{p-1}{p}} u^{-d-(\gamma-s)\frac{p-1}{p}}=u^{(\gamma-s)\frac{p-1}{p}} u^{\frac{p+\gamma-s}{p}}$, as well as $\zeta^{k-1}=\zeta^{\frac k{p'}+\frac kp-1}$, we get
\begin{align*}
\abs{\int_{B_{2R}(x_0)} u^{-d}|\nabla u|^{p-2}\nabla u\cdot \nabla \zeta^k \dx}\leq \frac{d}{2}\int_{B_{2R}(x_0)} |\nabla u|^{p} u^{\gamma-s} \zeta^k \dx 
+ C R^{-p}\int_{B_R(x_0)} \zeta^{k-p} u^{p+\gamma-s}\dx.
\end{align*}
Similarly, the last term in \eqref{ET1.8A} is estimated as
\begin{align*}
\abs{\int_{B_{2R}(x_0)} u^{-d}|\nabla u|^{p-2}\nabla u\cdot \nabla \zeta^k \dx}\leq \frac{d}{2}\int_{B_{2R}(x_0)} |\nabla u|^{q} u^{\gamma-s} \zeta^k \dx 
+ C R^{-q}\int_{B_R(x_0)} \zeta^{k-q} u^{q+\gamma-s}\dx
\end{align*}
with a suitable constant $C$. Inserting these estimate in \eqref{ET1.8A} we arrive at 
\begin{align}\label{ET1.8B}
&\int_{B_{2R}(x_0)}u^{\gamma}  \zeta^k dx + \frac{d}{2} \int_{B_{2R}(x_0)} |\nabla u|^{p} u^{\gamma-s}\zeta^k \dx + \frac{d}{2} \int_{B_{2R}(x_0)} |\nabla u|^{q} u^{\gamma-s}\zeta^k \dx\nonumber
\\
&\qquad \leq C R^{-p}\int_{B_{2R}(x_0)} \zeta^{k-p} u^{ \gamma-\gamma_p}\dx + C R^{-q}\int_{B_{2R}(x_0)} \zeta^{k-q} u^{\gamma-\gamma_q}\dx,
\end{align}
where 
$\gamma_p=s-p$ and $\gamma_q=s-q$. Clearly, $0<\gamma_p<\gamma_q < s-1$. Now, we consider the following three situations.

\noindent{\bf Case 1.} Suppose that $\gamma_q<\gamma$. 

Note that,
$$R^{-p}\zeta^{k-p}u^{\gamma-\gamma_p}=u^{\gamma-\gamma_p}\zeta^{\frac{k(\gamma-\gamma_p)}\gamma}R^{-p}\zeta^{k\frac{\gamma_p}\gamma-p},$$
so that, since $\gamma>\gamma_p$, applying Young's inequality with the pair $(\frac{\gamma}{\gamma-\gamma_p}, \frac{\gamma}{\gamma_p})$ we have
\begin{equation}\label{EL1.8B}
R^{-p}\int_{B_{2R}(x_0)} \zeta^{k-p} u^{\gamma-\gamma_p}\dx
\leq \frac{1}{4} \int_{B_{2R}(x_0)} u^\gamma \zeta^k\dx + C R^{-\frac{p\gamma}{\gamma_p}} \int_{B_{2R}(x_0)}  \zeta^{k-\frac{p\gamma}{\gamma_p}}\dx.
\end{equation}
Analogously, applying Young's inequality with the pair $(\frac{\gamma}{\gamma-\gamma_q}, \frac{\gamma}{\gamma_q})$ we have
$$ 
\int_{B_{2R}(x_0)} R^{-q}\zeta^{k-q} u^{q+\gamma-s}\dx
\leq \frac{1}{4} \int_{B_{2R}(x_0)} u^\gamma \zeta^k\dx + C R^{-\frac{q\gamma}{\gamma_q}} \int_{B_{2R}(x_0)}  \zeta^{k-\frac{q\gamma}{\gamma_q}}\dx.
$$
Set $k=\frac{p\gamma}{\gamma_p}$. This, of course, implies $k>\frac{q\gamma}{\gamma_q}$. Inserting the above two estimates into  \eqref{ET1.8B} yields that
$$ \int_{B_R(x_0)}u^{\gamma}  \dx\leq C (R^{N-\frac{p\gamma}{s-p}}+ R^{N-\frac{q\gamma}{s-q}})\leq 2C R^{N-\frac{q\gamma}{s-q}}.$$
Hence,  \eqref{EL1.8A} follows. 

\noindent{\bf Case 2.} Suppose that $\gamma=\gamma_q$.

Therefore, $\gamma>\gamma_p$ and thus \eqref{EL1.8B} continue to hold.   On the other hand, since now $ q+\gamma-s=0$ and being $k>q$ by the choice of $k$, then \eqref{ET1.8B} gives us
$$ \int_{B_{2R}(x_0)}u^{\gamma}  \zeta^k \dx\leq C (R^{N-\frac{p\gamma}{s-p}} + R^{N-q}).$$
Since $\frac{p\gamma}{s-p}>\frac{q\gamma}{s-q}=q$, we have $N-\frac{p\gamma}{s-p}\leq N-q$ and \eqref{EL1.8A} follows.

\noindent{\bf Case 3.} Suppose that $\gamma<\gamma_q$. 

In this case  applying  directly H\"{o}lder inequality with exponent $\frac{\gamma_q}{\gamma}$ and $\frac{\gamma_q}{\gamma_q-\gamma}$ we obtain
$$ \int_{B_R(x_0)} u^\gamma \dx
\leq C R^{N(1-\frac{\gamma}{\gamma_q})} \left(\int_{B_R(x_0)} u^{\gamma_q} \dx \right)^{\frac{\gamma}{\gamma_q}}
\leq C R^{N(1-\frac{\gamma}{\gamma_q})} \left(\int_{B_{2R}(x_0)} u^{\gamma_q}\zeta^k \dx \right)^{\frac{\gamma}{\gamma_q}}.
$$
Now from Case 2, we see that
$$ 
\int_{B_R(x_0)} u^\gamma \dx\leq C  R^{N(1-\frac{\gamma}{\gamma_q})} R^{(N-q)\frac{\gamma}{\gamma_q}}
= C R^{N-\frac{q\gamma}{s-q}}.
$$
This completes the proof.
\end{proof}

{\bf Proof of Theorem~\ref{t:s1}}:
\begin{proof}
(i) 
{\bf Case 1}: $s<q_*$.

Let $\Om$ be an exterior domain, that is, $\Om\supset\{|x|>R\}$ for some fixed $R>0$ and let $u$ be a positive solution to\eqref{main1}. Now we choose a sequence $x_j\in\Rn$ such that $|x_j|>3R$ and $|x_j|\to\infty$ as $j\to\infty$.
In particular, $B_{2R}(x_j)\subset\Omega$. As before, we set that $\gamma_q=s-q$. As  before, by $q>1$, it follows $0<\gamma_p<s-1$. We choose $$\gamma=\gamma_q,\qquad  B:=B_\frac{|x_j|}{4}(x_j)\subset\Om.$$
Then, from Lemma~\ref{L1.8}, we have
$$\int_{B}u^\gamma \dx\leq C|x_j|^{N-q}, \quad j=1,2,\cdots.$$
Therefore, 
$$\min_{\bar B}u^\gamma\leq\fint u^\gamma \dx\leq C|x_j|^{-q}, \quad j=1,2,\cdots.$$
Now we choose $y_j\in \bar{B}$ such that $\min_{\bar B}u^\gamma=u^\gamma(y_j)$.  This implies 
$$u^\gamma(y_j)\leq \frac{C}{|x_j|^q}.$$
Since, $y_j\in B$, it follows $\frac{3}{4}|x_j|\leq|y_j|\leq\frac{5}{4}|x_j|$. Therefore,
$$u^\gamma(y_j)\leq \frac{C}{|y_j|^q}.$$
Combining the above inequality with Lemma~\ref{L1.5}(i) with $\theta=\frac{N-q}{q-1}$, we have
$$\frac{C}{|y_j|^{\gamma\frac{N-q}{q-1}}}\leq\frac{C}{|y_j|^q}.$$
This yields a contradiction as $|y_j|\to\infty$ since $\gamma\,\frac{N-q}{q-1}= \gamma_q\,\frac{N-q}{q-1}<q$ by $s<q_*$. 

\smallskip

{\bf Case 2}: $s=q_*$.

Let $u$ be a positive solution to \eqref{main1} in the exterior domain $\Om\supset\{|x|>R_0\}$. We claim the following:

{\bf Claim:} Let $\mu_1\in(0, p(s-1)/s)$ and $\mu_2\in (0, q(s-1)/s)$. Then there exist a constant $C=C(N, p,q,s,\mu_1, \mu_2)>0$  such that for every ball
$B_{4R}(x_0)\subset\Om$ we have
\begin{align}
\int_{B_R(x_0)} |\na u|^{\mu_1} {\rm d}\theta &\leq CR^{N-\frac{qs\mu_1}{p(s-q)}}, 
\label{sphere}
\\
\int_{B_R(x_0)} |\na u|^{\mu_2} {\rm d}\theta &\leq CR^{N-\frac{s\mu_2}{(s-q)}}, 
\label{sphere1}
\end{align}
where ${\rm d}\theta$ is the surface measure on $\mathbb{S}^{N-1}$. 

We begin by proving the first inequality, and for simplicity, we denote $\mu_1=\mu$.
 
 Fix $\mu\in (0, p(s-1)/s)$ and  choose $\gamma\in (s-q, s-1)$ with $\gamma$  suitably near to $s-1$ so that $$\bar\gamma=\frac{(s-\gamma)\mu}{p-\mu}\in (0, s-1).$$ 
 Note that from \eqref{ET1.8B}, we have
\begin{equation}\label{E3.7}
\int_{B_{R}(x_0)} |\nabla u|^{p} u^{\gamma-s} \dx\leq  C R^{-p}\int_{B_{2R}(x_0)} \zeta^{k-p} u^{p+\gamma-s}\dx + C R^{-q}\int_{B_{2R}(x_0)} \zeta^{k-q} u^{q+\gamma-s}\dx
\end{equation}
for some constant $C$ and for all  ${B(x_0, 2R)}$ such that $B_{4R}(x_0)\subset \Omega$. 

Since $p+\gamma-s\in (0, \gamma)\subset (0, s-1)$,  indeed $\gamma>s-p$ by $p>q$, 
 from \eqref{EL1.8A} we obtain
\begin{align*}
R^{-p}\int_{B_{2R}(x_0)} \zeta^{k-p} u^{p+\gamma-s}\dx\leq R^{-p}\int_{B_{2R}(x_0)} u^{p+\gamma-s}\dx &\leq C R^{N-p-\frac{q(p+\gamma-s)}{s-q}}
\\
&\leq C R^{N-q-\frac{q(q+\gamma-s)}{s-q}} = CR^{N-\frac{q\gamma}{s-q}}.
\end{align*}
Analogously, 
$$R^{-q}\int_{B_{2R}(x_0)} \zeta^{k-q} u^{q+\gamma-s}\dx\leq CR^{N-\frac{q\gamma}{s-q}}.$$
Inserting these estimates in \eqref{E3.7} gives us
$$ \int_{B_{R}(x_0)} |\nabla u|^{p} u^{\gamma-s} \dx\leq C R^{N-\frac{q\gamma}{s-q}}.$$
Applying H\"{o}lder inequality and \eqref{EL1.8A} we arrive at
\begin{align*}
\int_{B_{R}(x_0)} |\nabla u|^{\mu}  \dx\leq \left(\int_{B_{R}(x_0)} |\nabla u|^{p} u^{\gamma-s} \dx\right)^{\frac{\mu}{p}}\left(\int_{B_{R}(x_0)} u^{\bar\gamma}\dx\right)^{1-\frac{\mu}{p}}
&\leq C R^{N-\frac{\gamma q}{s-q}\frac{\mu}{p}-\frac{\bar\gamma q}{s-q}\frac{p-\mu}{p}}
\\
& = CR^{N-\frac{qs\mu}{p(s-q)}},
\end{align*}
namely \eqref{sphere} with $\mu_1=\mu$.

To establish \eqref{sphere1}
we use the third integral in \eqref{ET1.8B} and repeat the above proof  by choosing $\bar\gamma=\frac{(s-\gamma)\mu_2}{q-\mu_2}\in (0, s-1)$. Hence the claims follow.

With these estimates in hand, we can apply the argument of \cite[Lemma~2.5]{SZ}.

Set $\mu_1=p-1$ and $\mu_2=q-1$. Since $s=q_*=\frac{q(N-1)}{N-q}$, we get $\frac{s}{s-q}=\frac{N-1}{q-1}$. Again, $p>q$ implies $\frac{p}{p-1}<\frac{q}{q-1}\Rightarrow \frac{q(p-1)}{p(q-1)}>1$. Thus, from \eqref{sphere} and \eqref{sphere1} we obtain that
\begin{equation*}
\int_{B_R(x_0)} |\na u|^{p-1} + |\na u|^{q-1} {\rm d}x\leq C R,
\end{equation*}
whenever $B_{4R}(x_0)\subset \Om$. Using a simple covering argument (we can cover $B_{9R}(0)\setminus B_{8R}(0)$ by a finite number of
balls $B_{2R}(y_i)$ with $|y_i|=9R$) this leads to 
\begin{equation*}
\int_{B_{9R}(0)\setminus B_{8R}(0)} |\na u|^{p-1} + |\na u|^{q-1} {\rm d}x\leq C R,
\end{equation*}
for some constant $C$ and all $R>R_0$. For an increasing 
sequence $K_j\to \infty$, we write the above inequality in the polar coordinate as
\begin{equation*}
\int_{8K_j}^{9K_j} \bigg(\int_{\mathbb{S}^N}
(|\na u(t,\theta)|^{p-1} + |\na u(t, \theta)|^{q-1} )
{\rm d}\theta\bigg)t^{N-1} {\rm d}t \leq C K_j,
\end{equation*}
and applying mean-value theorem for integral, we find $R_j\in (8K_j, 9K_j)$ such that

$$K_j R_j^{N-1}\int_{\mathbb{S}^N}
(|\na u(R_j,\theta)|^{p-1} + |\na u(R_j, \theta)|^{q-1}) {\rm d}\theta\leq C K_j.$$
Therefore, there exists a sequence $R_j\to\infty$ such that
\begin{equation}\label{sphere2}
\int_{\mathbb{S}^N}
(|\na u(R_j,\theta)|^{p-1} + |\na u(R_j, \theta)|^{q-1} )
{\rm d}\theta \leq C R^{-(N-1)}_j 
\end{equation}
for all $j=1, 2,\ldots$.

\medskip Now we can complete the proof for $s=q_*$. We integrate \eqref{main1} on $B_{R_j}(0)\setminus B_{R_0}(0)$ we obtain
\begin{align}\label{E3.8}
\int_{B_{R_j}(0)\setminus B_{R_0}(0)} u^{q_*-1}\dx &\leq -\int_{\partial B_{R_j}(0)} |\nabla u|^{p-2}\nabla u\cdot \nu_1 {\rm d}\theta 
-\int_{\partial B_{R_j}(0)} |\nabla u|^{q-2}\nabla u\cdot \nu_1 {\rm d}\theta\nonumber
\\
&\quad -\int_{\partial B_{R_0}(0)} |\nabla u|^{p-2}\nabla u\cdot \nu_2 {\rm d}\theta -\int_{\partial B_{R_0}(0)} |\nabla u|^{q-2}\nabla u\cdot \nu_2 {\rm d}\theta,
\end{align}
where $\nu_1, \nu_2$ denote unit outward normal on $\partial B_{R_j}(0)$ and $\partial B_{R_0}(0)$, respectively. From Lemma~\ref{L1.5}(i), with $\theta=\frac{N-q}{N-1}$ we note that
$$u^{q_*-1}(x)\geq \kappa |x|^{-N}\quad \text{in}\; B^c_{R_0}(0)$$
for some $\kappa>0$, and therefore, the LHS of \eqref{E3.8} goes to infinity as $R_j\to \infty$. We show that the RHS of \eqref{E3.8} is bounded
, uniformly in $j$, which leads to a contradiction. This would complete the proof.

Again, to establish boundedness of the RHS of \eqref{E3.8}, it is enough to show that 
\begin{equation}\label{E3.9}
\int_{\partial B_{R_j}(0)} |\nabla u|^{p-1} {\rm d}\theta+ \int_{\partial B_{R_j}(0)} |\nabla u|^{q-1} {\rm d}\theta \leq C
\end{equation}
for all $j$. Using \eqref{sphere2} we see that
\begin{align*}
\int_{\partial B_{R_j}(0)} |\nabla u|^{p-1} {\rm d}\theta+ \int_{\partial B_{R_j}(0)} |\nabla u|^{q-1} {\rm d}\theta
= R_j^{N-1} \int_{\mathbb{S}^N} 
(|\na u(R_j, \theta)|^{p-1}+|\na u(R_j, \theta)|^{q-1})
{\rm d}\theta\leq C.
\end{align*}
Hence the proof.




\medskip

(ii) If $1<q<s\leq p$ and $u$ is a bounded positive solution, then we choose $\gamma=\gamma_q=s-q$ in \eqref{ET1.8B}. This yields us 
\begin{equation}\label{27Aug1}
\int_{B_{2R}(x_0)}u^{\gamma_q}  \zeta^k \dx\leq C(R^{N-p}+R^{N-q})\leq C'R^{N-q}.
\end{equation}
Then the theorem follows from Case 1 of part(i), since $s<q_*$ is in force. 

\end{proof}

\begin{lemma}\label{liminf} 
Let $g(t),\, t>0$ be a positive function with $\inf_{t\geq t_0}g(t)>0$ for all $t_0>0$. If $u$ is a nonnegative  solution,  to
$$\Depq \geq g(u)$$ in an exterior domain $\Om$, then $\lim\inf_{|x|\to\infty} u(x)=0$.
\end{lemma}
\begin{proof}
Suppose, $\lim\inf_{|x|\to\infty} u(x)=\eps>0$. Let $y_j\in\Om$ such that $y_j\to\infty$ and $u(y_j)\to\eps$. Define, $\gamma=\inf_{t>\frac{\eps}{2}}g(t)$. Then for sufficiently large $|x|$,
$$(\Depq)(x)\geq g(u(x))\geq\gamma.$$
Next, we define $w(x):=c|x|^\frac{q}{q-1}$, where $c>0$ will be chosen later and 
$$w_\eps(x):=2\eps-w(x-y_j).$$


From the graph of $w_\eps$, it's clear that $\exists R=R_{\eps, c}>0$ such that 
$$w_\eps>0\quad\text{ in }\, B_R(y_j), \quad w_\eps=0\quad\text{ on }\, \partial B_R(y_j), \quad w_\eps<0 \quad\text{ in }\,\overline{B_R(y_j)}^c.$$
On $\partial B_R(y_j)$, $w_\eps=0$ implies 
\begin{equation}\label{9Aug1}
2\eps=cR_{\eps, c}^\frac{q}{q-1}.
\end{equation}
If $x\in  B_R(y_j)$ then
\begin{align}\label{9Aug2}
(-\De_p w_\eps-\De_p w_\eps)(x)&=(\De_p w+\De_q w)(x-y_j)\nonumber\\
&<\underbrace{c^{p-1}\big(\frac{q}{q-1}\big)^{p-1}\bigg(N-1+\frac{p-1}{q-1}\bigg)R_{\eps,c}^\frac{p-q}{q-1}}_{(I)}+\underbrace{c^{q-1}\big(\frac{q}{q-1}\big)^{q-1}N}_{(II)}.
\end{align}

We aim to choose $c>0$ and $R_{\eps, c}>0$ in such a way that $(I)<\frac{\gamma}{4}$ and 
 $(II)<\frac{\gamma}{4}$. To do this, we first substitute $c$ from \eqref{9Aug1} to $(I)$ and therefore we require
\begin{equation}\label{9Aug3}
\bigg(\frac{2\eps}{R_{\eps, c}^\frac{q}{q-1}}\bigg)^{p-1}AR_{\eps, c}^\frac{p-q}{q-1}<\frac{\gamma}{4},
\end{equation}
 where $A=\big(\frac{q}{q-1}\big)^{p-1}\big(N-1+\frac{p-1}{q-1}\big)$. \eqref{9Aug3} is equivalent to 
 \begin{equation}\label{9Aug4}
(2\eps)^{p-1}AR_{\eps, c}^{-p}<\frac{\gamma}{4}.
\end{equation}
Since $\gamma, \eps, A$ all are fixed constants, by choosing $R_{\eps,c}>>1$ we can achieve \eqref{9Aug3}. Once we have fixed $R_{\eps,c}$ and substitute in \eqref{9Aug1}, $c>0$ gets fixed. Also observe that $R>>1$ means $c<<1$. Therefore, while choosing $R_{\eps,c}$ we make sure that \eqref{9Aug4} is satisfied and $(II)<\frac{\gamma}{4}$. Hence, RHS of \eqref{9Aug2} becomes less than $\frac{\gamma}{2}$. Therefore,
$$-\De_p w_\eps-\De_p w_\eps\leq \Depq \quad\text{ in }\, B_{R_{\eps, c}}(y_j), \quad w_\eps=0<u \quad\text{ on }\, \partial B_{R_{\eps, c}}(y_j).$$
Thus applying Lemma~\ref{l:comp1} we obtain $w_\eps\leq u$ in $B_{R_{\eps, c}}(y_j)$. In particular, $w_\eps(y_j)\leq u(y_j)$, that is,  $2\eps\leq u(y_j)$.  This yields a contradiction to the definition of $\eps$ by choosing $y_j$ sufficiently large. Hence the lemma is proved.  
\end{proof}

\medskip{\bf Proof of Theorem~\ref{t:Nq}.}
Let $N\leq q$ and let $u$ be a positive solution of \eqref{main1}. Then, by Lemma~\ref{liminf}, we have
$\liminf_{|x|\to\infty} u(x)=0$. On the other hand, as $u$ becomes 
$(p,q)-$super harmonic function by Lemma~\ref{L1.5}(ii), it also follows
$\lim\inf_{|x|\to\infty} u(x)>0$. This yields a contradiction. Hence the theorem follows. 
\hfill $\Box$

\section{Gradient nonlinear terms}
In this section, we prove Theorem~\ref{T1.6}.

\medskip {\bf Proof of Theorem~\ref{T1.6}.}
First, we assume
\begin{equation}\label{ET1.8AB}
    s(N-q)+m(N-1)<N(q-1).
\end{equation}
Let $\zeta\in C^2_c(B(0,1))$ be a radial cut-off function such that $\zeta\equiv 1$ in $B(0,\frac{1}{2})$, $|\nabla \zeta|\leq 2$, $|D^2\zeta|\leq C$. For $R>2$, we define $\varphi:=\zeta^k\big(\frac{|x-x_0|}{R}\big)$, where $x_0$ has been chosen arbitrarily and fixed and $k>0$ will be chosen later. Taking $\varphi$ as the test function for the above inequality we obtain
\begin{equation}\label{ET1.6B}
\int_{B_R}u^s |\nabla u|^m \varphi \dx \leq\int_{B_R} |\nabla u|^{p-2}\nabla u\cdot \nabla \varphi \dx + \int_{B_R} |\nabla u|^{q-2}\nabla u\cdot \nabla \varphi\dx.
\end{equation}
As $m>p-1>q-1$, using Young's inequality, we compute
\begin{align*}
\int_{B_R} |\nabla u|^{q-2}\nabla u\cdot \nabla \zeta^k \dx &\leq  \frac{k}{R} \int_{B_R} |\nabla u|^{q-1} \zeta^{k-1} |\nabla \zeta| \dx
\\
&= k\int_{B_R} |\nabla u|^{q-1} u^{s\frac{q-1}{m}} \zeta^{k(\frac{q-1}{m}+\frac{m-q+1}{m})} u^{-s\frac{q-1}{m}} \frac{|\nabla \zeta|}{R\zeta} \dx
\\
&\leq \frac{1}{4} \int_{B_R}u^s |\nabla u|^m \zeta^k \dx + 
C\int_{ A_R} \zeta^{k-\frac{m}{m-q+1}} u^{-s\frac{q-1}{m-q+1}} R^{-\frac{m}{m-q+1}} \dx,
\end{align*}
where $A_R$ denotes the annulus $\frac{R}{2}<|x|<R$. 
We set $k\geq \frac{m}{m-q+1}$. Thus $k> \frac{m}{m-p+1}$ as $p>q$. Also, set $\theta=\frac{N-q}{q-1}$ so that
$m\frac{N-1}{q-1} + s\theta<N$, by \eqref{ET1.8AB}. From Lemma~\ref{L1.5}, it then follows that
\begin{equation}\label{decay}
u(x)\geq \frac{C}{|x|^{\theta}}\quad\text{in }\, |x|>1.
\end{equation}
Therefore, assuming $R>1$
$$ \int_{B_R} |\nabla u|^{q-2}\nabla u\cdot \nabla \varphi\leq  \frac{1}{4} \int_{B_R}u^s |\nabla u|^m \zeta^k \dx
+ C R^{N-\frac{m}{m-q+1}+s\theta\frac{q-1}{m-q+1}}.
$$
A similar argument also gives us 
\begin{equation*}
\int_{B_R} |\nabla u|^{p-2}\nabla u\cdot \nabla \zeta^k\leq  \frac{1}{4} \int_{B_R}u^s |\nabla u|^m \zeta^k \dx
+ C R^{N-\frac{m}{m-p+1}+s\theta\frac{p-1}{m-p+1}}.
\end{equation*}
Inserting the above estimates in \eqref{ET1.6B} we arrive at
\begin{equation}\label{ET1.6C}
 \int_{B_{\frac{R}{2}}}u^s |\nabla u|^m  \dx\leq C R^{N+\frac{s\theta(p-1)-m}{m-p+1}}+ CR^{N+\frac{s\theta(q-1)-m}{m-q+1}}
\end{equation}
for all $R>2$. By our choice of $\theta$ we note that
$$N+\frac{s\theta(q-1)-m}{m-q+1}<0 \Longleftrightarrow m \frac{N-1}{q-1} + s\theta <N,$$
and since $m \frac{N-1}{p-1} + s\theta< m \frac{N-1}{q-1} + s\theta <N$, we also have $N+\frac{s\theta(p-1)-m}{m-p+1}<0$. Thus, letting
$R\to\infty$ in \eqref{ET1.6C}, we get $u$ to be a constant.

Next we suppose
$$s(N-q)+m(N-1)= N(q-1).$$
Letting $R\to\infty$ in \eqref{ET1.6C} it follows that
$$\int_{\mathbb R^N} u^s |\nabla u|^m \dx <\infty.$$
Therefore , there exists a sequence $R_j\to\infty$, such that 
$$\int_{B_{R_j}^c} u^s |\nabla u|^m \dx \to 0\quad 
\text{as}\;\; R_j\to\infty.$$
Now consider a sequence of cut-off functions $\varphi_j$ satisfying
$$\varphi_j=1\quad \text{in}\; B_{R_j}, \quad \varphi_j=0\quad \text{in}\; B^c_{2R_j}, \quad \text{and}
\quad \sup|\nabla \varphi_j|\leq \frac{C}{R_j},$$
where $C$ is independent of $R_j$. Using \eqref{ET1.6B}
we write
\begin{equation}\label{base-eq}
\begin{aligned}
\int u^{s}|\nabla u|^m\varphi_j \, \dx&\le \int_{\mathbb R^N}|\nabla u|^{p-1}|\nabla\varphi_j| \, \dx+ \int_{\mathbb R^N} |\nabla u|^{q-1}|\nabla\varphi_j| \, \dx.
\end{aligned}
\end{equation}
Let $A_{R_j}=B_{2R_j}\setminus B_{R_j}$ be the annulus where $\nabla\varphi_j$ is supported. Then , using \eqref{decay}, we estimate
$$
\begin{aligned}
 \int_{\mathbb R^N}|\nabla u|^{p-1}|\nabla\varphi_j| \, \dx 
 & = \int_{A_{R_j}} u^{-\frac{s(p-1)}{m}}|\nabla \varphi_j|\, u^{\frac{s(p-1)}{m}}|Du|^{p-1} \, \dx
 \\
 &\le C R^{\frac{N-q}{q-1}\frac{s(p-1)}{m}-1}\int_{A_{R_j}}  u^{\frac{s(p-1)}{m}}|\nabla u|^{p-1} \, \dx
 \\
 &\leq C R^{\frac{N-q}{q-1}\frac{s(p-1)}{m}-1+ \frac{N(m-p+1)}{m}} \left(\int_{A_{R_j}}  u^{s}|\nabla u|^{m} \, \dx\right)^\frac{p-1}{m}
 \\
 &\le C R^{\frac{N-q}{q-1}\frac{s(p-1)}{m}-1+ \frac{N(m-p+1)}{m}} \left(\int_{B^c_{R_j}}  u^{s}|\nabla u|^{m} \, \dx\right)^\frac{p-1}{m}.
\end{aligned}
$$
Using the relation $s(N-q)+m(N-1)= N(q-1)$, we see that 
$$
\frac{N-q}{q-1}\cdot\frac{s(p-1)}{m}-1+ \frac{N(m-p+1)}{m}= \frac{(N-1)(q-p)}{q-1}<0.
$$
A similar calculation also gives 
\begin{align*}
 \int_{\mathbb R^N}|\nabla u|^{q-1}|\nabla \varphi_j| \, \dx\leq  C   \left(\int_{B^c_{R_j}}  u^{s}|\nabla u|^{m} \, \dx\right)^\frac{q-1}{m}.
\end{align*}
Putting them in \eqref{base-eq}, we obtain
$$
\int u^{s}|\nabla u|^m\varphi_j \, \dx\leq C \left(\int_{B^c_{R_j}}  u^{s}|\nabla u|^{m} \, \dx\right)^\frac{p-1}{m} + 
C   \left(\int_{B^c_{R_j}}  u^{s}|\nabla u|^{m} \, \dx\right)^\frac{q-1}{m}.
$$
Now, by the property of $R_j$, the right-hand side above goes to $0$, proving $u$ to be a constant.

\hfill $\Box$

\medskip
 Note that for $s=0$, we do not require $u$ to be positive in the argument above  and in addition no use of Lemma~\ref{L1.5} is needed.
As an immediate corollary to the proof Theorem~\ref{T1.6}, we have (compare it with \cite[Corollary~1.3]{Gof} where a similar result is
established for the $p$-Laplacian)
\begin{cor}\label{Cor5.1}
Let $N\geq 2$,  $1<q<N$ and  $p-1< m\leq \frac{N(q-1)}{N-1}$. Assume $u$ is any solution of the inequality 
 $$ \Depq \geq  |\nabla u|^m \quad \text{ in }\Rn.$$ 
Then $u$ is constant.
\end{cor}

To prove Theorem~\ref{T1.7} we need the following improvement on the lower bound (compare with Lemma~\ref{L1.5}).

{\bf Proof of Lemma~\ref{L5.1}.}
We note that by Lemma \ref{L1.5}, estimate \eqref{L5.1B} holds for any
$\theta\ge\frac{N-q}{q-1}$, hence it is enough to prove 
\eqref{L5.1B} only for 
$\theta\in (0, \frac{N-q}{q-1}\bigl)$.
To this aim, we use an iteration process which we describe next.

Let  $0<\sigma\leq \frac{N-q}{q-1}$  and suppose that
for some constant $C_{\sigma}$
we have
\begin{equation}\label{L5.1C}
u(x)\geq C_\sigma |x|^{-\sigma}\quad \text{in}\; \bar\Om\cap\{|x|\geq 1\}.
\end{equation}
This clearly holds for $\sigma=\frac{N-q}{q-1}$, by Lemma~\ref{L1.5}. Using \eqref{L5.1C}, we then have
\begin{equation}\label{L5.1D}
\Depq \geq \tilde{C}_\sigma |x|^{-s\sigma}|\na u|^m \quad \text{in}\;\; \Om\cap\{|x|\geq 1\}
\end{equation}
for $\tilde{C}_\sigma=(C_\sigma)^s$.

Then, take $\sigma_1$ such that $0<\sigma_1<\sigma$.
We claim that, for $\varphi_{\sigma_1}(x)= c|x|^{-\sigma_1}$, where $c$ is a positive constant, if we have
\begin{align}\label{L5.1E}
\max\{|\Delta_p\varphi_{\sigma_1}|, |\Delta_q\varphi_{\sigma_1}|\}<\frac{\tilde{C}_\sigma}{2} |x|^{-s\sigma}|\na \varphi_{\sigma_1}|^m
\quad\text{for } |x|\geq r_{\sigma_1},
\end{align}
and some $r_{\sigma_1}>1$, then for some constant $C_{\sigma_1}$ we have
\begin{equation}\label{L5.1F}
u(x)\geq C_{\sigma_1} |x|^{-\sigma_1}\quad \text{in}\;\; \bar\Om\cap\{|x|\geq 1\}.
\end{equation}
To prove this claim, since $u$ is positive in $\bar\Om$, it is enough to prove \eqref{L5.1F} for $|x|\geq r_{\sigma_1}$. We also set $r_{\sigma_1}$ large enough so that 
$\{|x|\geq r_{\sigma_1}\} \subset \Om$. Since $m\leq q-1$, from \eqref{L5.1E} we get for $\kappa\in (0, 1)$ that
$$-\Delta_p (\kappa \varphi_{\sigma_1}) -
\Delta_q (\kappa \varphi_{\sigma_1})
\leq \kappa^{p-1} |\Delta_p \varphi_{\sigma_1}| + 
\kappa^{q-1} |\Delta_q \varphi_{\sigma_1}|
< \tilde{C}_\sigma |x|^{-s\sigma} |\na (\kappa \varphi_{\sigma_1})|^m$$
for $|x|\geq r_{\sigma_1}$. Choose $\kappa>0$ small enough so that
$u(x)\geq \kappa \varphi_{\sigma_1}(x)$ on $|x|=r_{\sigma_1}$.
Let $\varepsilon\in (0, 1)$ be small and define $\psi=\kappa\varphi_{\sigma_1}-\varepsilon$. Then
\begin{equation}\label{L5.1G}
-\Delta_p \psi -
\Delta_q \psi< \tilde{C}_\sigma |x|^{-s\sigma} |\na \psi|^m
\quad \text{for}\; \; |x|\geq r_{\sigma_1}.
\end{equation}
Thus $\psi$ is a strict subsolution of $-\Delta_p v-\Delta_q v = \tilde{C}_\sigma |x|^{-s\sigma} |\na v|^m $ in $\{|x|\geq r_{\sigma_1}\}$. 

Next we show that $u\geq \psi$ in   $\{|x|\geq r_{\sigma_1}\}$ for all $\varepsilon\in (0, 1)$. Then, letting $\varepsilon\to 0$, we get
\eqref{L5.1F}. To prove the comparison (that is, $u\geq \psi$ in $\{|x|\geq r_{\sigma_1}\}$), we suppose, to the contrary, that $u(z)<\psi(z)$
for some $|z|>r_{\sigma_1}$. Note that, by our choice of $\kappa$, 
we have $u> \psi$ for $|x|=r_{\sigma_1}$. Again, since $\psi$ is negative for large $x$, we can find a positive constant $\kappa_1$
such that $\psi-\kappa_1\leq u$ in $\{|x|\geq r_{\sigma_1}\}$ and
for some $\bar{z}$, with  $|\bar z|>r_{\sigma_1}$, we have
$\psi(\bar{z})-\kappa_1=u(\bar{z})$. In particular, 
$\na u(\bar{z})=\na \psi(\bar z)\neq 0$. By
Lemma~\ref{L-vis} and \eqref{L5.1D},  we then have
$$
-\Delta_p \psi(\bar{z}) -\Delta_q \psi(\bar{z})=
-\Delta_p(\psi(\bar{z})-\kappa_1) -\Delta_q (\psi(\bar{z})-\kappa_1)
\geq \tilde{C}_\sigma |\bar{z}|^{-s\sigma} |\na \psi(z)|^m,
$$
which contradicts \eqref{L5.1G}. This gives us the comparison. Thus, {\it we have established that \eqref{L5.1C} and \eqref{L5.1E} give
\eqref{L5.1F}}.

Now we complete the proof of this lemma in three cases.

\noindent{\it Case 1.} Suppose $m=q-1$. Set $\sigma=\frac{N-q}{q-1}$. Therefore, from \eqref{L5.1A}, it follows that  $s\sigma <1$. 

By our choice of $\sigma$, then \eqref{L5.1C} holds from Lemma~\ref{L1.5}. Choose  $\theta\in (0, \sigma)$ and let
$\varphi_\theta(x)=|x|^{-\theta}$. It is easily seen that
(see Lemma~\ref{L1.5})
\begin{align*}
 -\De_p\varphi_\theta(x) & = \theta^{p-1} |x|^{-(\theta+1)(p-1)-1}\big(N-(\theta+1)(p-1)-1\big),
\\
-\De_q\varphi_\theta(x) & = \theta^{q-1} |x|^{-(\theta+1)(q-1)-1}\big(N-(\theta+1)(q-1)-1\big),
\\
|x|^{-s\sigma}|\na\varphi_\theta(x)|^m &=\theta^m |x|^{-s\sigma-(\theta+1)m}
\end{align*}
for $x\neq 0$. Since
$$(\theta+1)(p-1)+1> (\theta+1)(q-1)+1>  (\theta+1)m + s\sigma,$$
choosing $\theta=\sigma_1$, then \eqref{L5.1E} holds.  
Therefore, \eqref{L5.1B}
follows.

\noindent{\it Case 2.} Suppose that $m<q-1$ and for some 
$\sigma\in (0, \frac{N-q}{q-1}]$ we have $s\sigma \leq q-m$  by \eqref{L5.1A}. Also, assume that \eqref{L5.1C} holds  (we'll prove it later). To check \eqref{L5.1E} we choose
$\theta=\sigma_1<\sigma$.
Since
\begin{align*}
(\theta+1)(p-1)+1> (\theta+1)(q-1)+1 &= (\theta+1)(q-m) + m(\theta+1)  -\theta
\\
&> s\sigma + m(\theta+1),
\end{align*}
\eqref{L5.1E} follows. In particular, if we have $q-m\geq s\frac{N-q}{q-1}$, we get \eqref{L5.1B}. 

\noindent{\it Case 3.} Suppose that $m<q-1$ and $s\frac{N-q}{q-1}> q-m$.
Define the function
$$\mathfrak{g}(t)=\frac{s t-1}{q-1-m}-1=\frac{s t}{q-1-m}-\frac{q-m}{q-1-m}.$$
Note that $\mathfrak{g}(t)>0$ for $st>q-m$. Set $\sigma_0=\frac{N-q}{q-1}$ and $\sigma_k=\mathfrak{g}(\sigma_{k-1})$ for $k=1, 2, \ldots$. Also,
\begin{align*}
\mathfrak{g}(\sigma_0)<\sigma_0
&\Leftrightarrow s\frac{N-q}{q-1} < \frac{(N-1)(q-1-m)}{q-1} + 1
\\
& \Leftrightarrow s(N-q) - (q-1) < (N-1)(q-1-m)
\\
& \Leftrightarrow s(N-q) + m(N-1)< N(q-1),
\end{align*}
where the last inequality is assured by \eqref{L5.1A}. Again, using
the linearity of $\mathfrak{g}$ we have
$$ \sigma_{k+1}-\sigma_k=\left(\frac{s}{q-1-m}\right)^k(\sigma_1-\sigma_0)<0.$$
Thus, $\{\sigma_k\}$ forms a strictly decreasing sequence. Since
$$(\sigma_{k+1} +1)(p-1) + 1> (\sigma_{k+1} +1)(q-1) + 1 
= s\sigma_k + (\sigma_{k+1} +1) m\quad \text{and} \quad m< q-1,
$$
as soon as $\sigma_{k+1}>-1$, then taking $c$ small enough we see that \eqref{L5.1E} hold for
$\varphi_{\sigma_{k+1}}(x)=c |x|^{-\sigma_{k+1}}$, provided  we set  $\sigma=\sigma_k$. 

We start by considering $k=0$, so that, our assumption $s\sigma_0>q-m$
gives $\sigma_1=g(\sigma_0)>0$, in turn the above inequality becomes 
$$(\sigma_{1} +1)(p-1) + 1> (\sigma_{1} +1)(q-1) + 1 
= s\sigma_0 + (\sigma_{1} +1) m,$$
yielding  \eqref{L5.1E} with 
$\sigma=\sigma_0$. In turn,
 \eqref{L5.1F} is in force, provided that \eqref{L5.1C}  holds with $\sigma=\sigma_0$.  In other words, \eqref{L5.1B} holds for $\theta=\sigma_1$.

Next,  to show  \eqref{L5.1B} for all $\theta$, we can continue this iteration to improve the lower bound of $u$ to arrive in the situation of Case 2. 

We claim that there exists
$m\geq 1$ such that $s\sigma_m\in (0, q-m]$ and \eqref{L5.1C} holds
for $\sigma=\sigma_m(<\sigma_0)$. Once we establish this, we are in the setup of Case 2 and therefore, \eqref{L5.1B} follows. To prove the last claim, first we suppose that $s/(q-1-m)<1$. In this case, $\mathfrak{g}$ is a contraction map, with a negative fixed point.
Since $\sigma_k$ tends to this fixed point, by contraction mapping theorem, there exists $m\geq 1$ satisfying 
$s\sigma_m\in (0, q-m]$ and $s \sigma_{m-1}>q-m$. Furthermore, when
$s/(q-1-m)\geq 1$, we get 
$$ \sigma_{k+1}-\sigma_k=\left(\frac{s}{q-1-m}\right)^k(\sigma_1-\sigma_0)< \sigma_1-\sigma_0<0.$$
Therefore, $\sigma_k\to -\infty$, otherwise if the limit would be finite  we reach immediately a contradiction  since at every iteration $\sigma_k$ decreases by a fixed constant. 
Thus, there exists $m\geq 1$ satisfying 
$s\sigma_m\in (0, q-m]$ and $s \sigma_{m-1}>q-m$. This proves our claim and completes the proof.
\hfill $\Box$

\medskip

{\bf Proof of Theorem~\ref{T1.7}.}
For the proof we use the same idea as in Theorem~\ref{t:qbig}.
Choose $\theta>0$ such that $s\theta<1$.
Let $\bar\theta\in (0, 1)$. Therefore
\begin{equation}\label{T5.1A}
(\bar\theta-1)( q-1-m) + s\theta< 
(\bar\theta-1)( p-1-m) + s\theta<1.
\end{equation}

For $ \varepsilon \in (0, 1)$ we now define
$\varphi_\varepsilon(x)=-\varepsilon |x|^{\bar\theta}$. 
Next we show that for given $\kappa>0$, there exists $r_0>0$ satisfying
\begin{equation}\label{T5.1B}
-\Delta_p\varphi_{\eps} -\Delta_q \varphi_{\eps}< \kappa|x|^{-s\theta} |\na \varphi_{\eps}|^m\quad \text{for}\; |x|\geq r_0,
\end{equation}
uniformly in $\varepsilon$. A straightforward calculation 
reveals that
\begin{align*}
&-\Delta_p\varphi_{\eps} -\Delta_q \varphi_{\eps} - 
\kappa|x|^{-s\theta} |\na \varphi_{\eps}|^m
\\
&= -\varepsilon^{m}|x|^{(\bar\theta-1)m-s\theta}
\Bigl[\bar\theta^{p-1} \varepsilon^{p-1-m} |x|^{(\bar\theta-1)(p-1-m) + s\theta-1}(N-1+(\bar\theta-1)(p-1))  
\\
&\qquad +\bar\theta^{q-1} \varepsilon^{q-1-m} |x|^{(\bar\theta-1)(q-1-m) + s\theta-1}(N-1+(\bar\theta-1)(q-1)) +\kappa \bar\theta^m \Bigr]
\\
&\leq \varepsilon^{m}|x|^{(\bar\theta-1)m-s\theta}
\Bigl[\bar\theta^{p-1} |x|^{(\bar\theta-1)(p-1-m) + s\theta-1}|N-1+(\bar\theta-1)(p-1)| 
\\
&\qquad +\bar\theta^{q-1} |x|^{(\bar\theta-1)(q-1-m) + s\theta-1}|N-1+(\bar\theta-1)(q-1)| -\kappa \bar\theta^m \Bigr].
\end{align*}
From \eqref{T5.1A} we note that the exponents of $|x|$ in the above expression are negative. Hence we can choose $r_0$ large enough, depending on $\bar\theta$ and $\kappa$, so that \eqref{T5.1B} holds.  Note that $r_0$ is independent of $\varepsilon$. 

 Now let $u$ be any positive solution of \eqref{grad non}, applying Lemma~\ref{L5.1} we see that
\begin{equation}\label{T5.1C}
\Depq \geq \kappa |x|^{-s\theta}|\na u|^m 
\end{equation}
for $|x|\geq 1$ and some $\kappa>0$. We use this $\kappa$ in 
\eqref{T5.1B} and adjust $r_0>1$ accordingly. Next we show that
\begin{equation}\label{T5.1D}
u(x)\geq \min_{\bar{B}_{r_0}} u:=\varrho_0  \quad\text{in } \mathbb R^N.
\end{equation}
To establish \eqref{T5.1D} we let $\tilde\varphi_{\eps}:=\varrho_0+\varphi_\varepsilon$. From \eqref{T5.1B} we see that
$$
-\Delta_p\tilde\varphi_{\eps} -\Delta_q \tilde\varphi_{\eps}< \kappa|x|^{-s\theta} |\na \tilde\varphi_{\eps}|^m\quad \text{for}\; |x|\geq r_0.
$$
Applying the comparison argument, as in Lemma~\ref{L5.1}  with $\psi$ replaced by  $\tilde\varphi_{\eps}$, we get 
$u\geq \tilde\varphi_{\eps}$ in $B^c_{r_0}$ for all $\varepsilon\in (0, 1)$. 
Letting $\varepsilon\to 0$, we obtain \eqref{T5.1D}.
 Since $u$ attains its minimum in $\bar{B}_{r_0}$,  using Lemma~\ref{R1} it then follows that $u\equiv \varrho_0$, which completes the proof.
\hfill $\Box$

\medskip

\begin{proof}[{\bf Proof of Theorem~\ref{t:s2}}.]
The proof is based on an argument of contradiction. Let $u$ be a nonnegative solution to \eqref{main1} with $1<s<q_*$. By Lemma~\ref{liminf}, we have $\liminf_{|x|\to\infty} u(x)=0$. 

Again, since $\Omega$ is an
exterior domain, applying Lemma~\ref{L5.1}  with $m=0$, we get that for any $\theta>0$ there exists $\kappa>0$ such that
$$ u(x)\geq \kappa |x|^{-\theta}\quad \text{for all}\; |x|\geq R_0.$$
From \eqref{main1} this leads to 
\begin{equation}\label{new1}
\Depq\geq \kappa^{s-1} |x|^{-\theta(s-1)}\quad \text{for}\; |x|\geq R_0.
\end{equation}
As done in the proof of Theorem~\ref{T1.7}, we choose $\theta>0$ such that $s\theta<1$ and $\bar\theta$ to satisfying \eqref{T5.1A}  with $m=0$ and with $(s-1)$
in place of $s$. Then, for $\varphi_\varepsilon(x)=-\varepsilon|x|^{\bar\theta}$, there exists $r_0>R_0$ such that
\begin{equation}\label{new2}
-\Delta_p\varphi_{\eps} -\Delta_q \varphi_{\eps}< \kappa^{s-1} |x|^{-\theta(s-1)} \quad \text{for}\; |x|\geq r_0,
\end{equation}
uniformly in $\varepsilon$. Let $\tilde\varphi_\varepsilon= \min_{|x|=r_0} u +\varphi_\varepsilon$. Note that $\tilde\varphi_\varepsilon$ is also a subsolution to the above equation. Now following the arguments of Theorem~\ref{T1.7} we find that
$u\geq \tilde\varphi_\varepsilon$ for $|x|\geq r_0$. Letting $\varepsilon\to 0$, we obtain $u(x)\geq \min_{|x|=r_0} u$ for all $|x|\geq r_0$.
In particular, $\liminf_{|x|\to\infty} u(x)>0$, which contradicts our conclusion of Lemma~\ref{liminf}. Therefore, no such supersolution $u$ could 
exist. Hence the proof.
\end{proof}

\medskip

\bigskip

\medskip

{\bf Funding:} This research of M.~Bhakta
is partially supported by DST Swarnajaynti fellowship (SB/SJF/2021-22/09), ANRF-ARG-MATRICS grant (ANRF/ARGM/2025/000126/MTR)and \\
INdAM-ICTP joint research in pairs program  for 2025. A.~Biswas is partially supported by a DST Swarnajaynti fellowship (SB/SJF/2020-21/03).
Filippucci is a member of the {\em Gruppo Nazionale per l'Analisi Ma\-te\-ma\-ti\-ca, la Probabilit\`a e le loro Applicazioni} (GNAMPA) of the {\em Istituto Nazionale di Alta Matematica} (INdAM) and was  partly supported  by
 INdAM-ICTP joint research in pairs program for 2025.
M.~Bhakta and R. Filippucci would like to thank the warm hospitality of ICTP,  the travel support and daily allowances provided by INdAM-ICTP.  

\medskip

{\bf Data availability:} Data sharing not applicable to this article as no datasets were generated or analyzed during the current study.

\medskip

{\bf Conflict of interest} The authors have no Conflict of interest to declare that are relevant to the content of this article.

\end{document}